\documentclass{l4dc2025}

\title[Linear Convergence of Single-loop Algorithm for Bilevel Optimization]{Linear Convergence Analysis of Single-loop Algorithm for Bilevel Optimization via Small-gain Theorem}
\usepackage{times}
\usepackage{amsfonts, amsmath, amssymb}
\usepackage{threeparttable}
\usepackage{arydshln}
\usepackage{bm}
\usepackage{ulem}
\usepackage{breqn}

\newtheorem{theorem}{Theorem}
\newtheorem{lemma}{Lemma}
\newtheorem{remark}{Remark}
\newtheorem{assumption}{Assumption}

\author{%
 \Name{Jianhui Li} \Email{jianhuili@zju.edu.cn}\\
 \addr College of Control Science and Engineering, Zhejiang University, P. R. China\\
  \addr School of Data Science, The Chinese University of Hong Kong, Shenzhen, P. R. China
 \AND
 \Name{Shi Pu} \Email{pushi@cuhk.edu.cn}\\
 \addr School of Data Science, The Chinese University of Hong Kong, Shenzhen, P. R. China
 \AND
 \Name{Jianqi Chen} \Email{jqchen@nju.edu.cn}\\
 \addr Center for Advanced Control and Smart Operations (CACSO), Nanjing University, P. R. China
 \AND
 \Name{Junfeng Wu} \Email{junfengwu@cuhk.edu.cn}\\
 \addr School of Data Science, The Chinese University of Hong Kong, Shenzhen, P. R. China
}

\begin{document}

\maketitle

\begin{abstract}%
Bilevel optimization has gained considerable attention due to its broad applicability across various fields. While several studies have investigated the convergence rates in the strongly-convex-strongly-convex (SC-SC) setting, no prior work has proven that a single-loop algorithm can achieve linear convergence. This paper employs a small-gain theorem in {robust control theory} to demonstrate that {a} single-loop algorithm based on the implicit function theorem attains a linear convergence rate of $\mathcal{O}(\rho^{k})$, where $\rho\in(0,1)$ is specified in Theorem~\ref{theorem:final_convergence}.  {Specifically, We model the algorithm as a dynamical system by identifying its two interconnected components: the controller (the gradient or approximate gradient functions) and the plant (the update rule of variables)}. We prove that each component exhibits a bounded gain {and that}, with carefully designed step sizes,  {their cascade accommodates a product gain} strictly less than {one}. Consequently,  {the overall algorithm can be proven to achieve a linear convergence rate, as guaranteed by the small-gain theorem}. The gradient boundedness assumption adopted in the single-loop algorithm (\cite{hong2023two, chen2022single}) is replaced with a gradient Lipschitz assumption in Assumption \ref{assumption_upper}.\ref{assumption_upper_smooth}. To the best of our knowledge, this work is first-known result on linear convergence for  {a} single-loop algorithm.
\end{abstract}


\section{Introduction}
\label{sec:introduction}
Bilevel optimization has witnessed an up-soaring in recent years due to its broad applicability to various {domains}, such as adversarial training (\cite{huang2020metapoison, robey2023adversarial}), hyper-parameter tuning (\cite{franceschi2018bilevel, sinha2020gradient, okuno2021lp}), {and} model pruning (\cite{zhang2022advancing, sehwag2020hydra}). This class of optimization {problems} is characterized by two nested levels, namely an upper and a lower level, where the decision variables are {mutually constrained}. {Specifically,} a bilevel optimization problem can be formalized as follows:
\begin{equation}
    \label{equ:BLO-problem}
    \begin{aligned}
        \text{minimize}_{\omega} \quad & f(\omega, v)&\textit{(upper-level problem)}\\
        \text{s.t.} \quad & v \in \mathop{\arg\min}_{v} g(\omega,v) \quad & \textit{(lower-level problem),}
    \end{aligned}
\end{equation}
where $f$ and $g$ are continuously differentiable functions, and $\omega\in\mathbb{R}^m, v\in\mathbb{R}^n$ are the upper- and lower-level variables, respectively. \par

\begin{table}[t]
    \centering
    \caption{Summary of convergence rates of bilevel optimization algorithms. {``Design'' refers} to whether the algorithm adopts a single-loop or double-loop structure. We use the abbreviations SC-SC for strongly-convex-strongly-convex  and NC-SC for nonconvex-strongly-convex conditions.  For the specific {values of} $\rho_1$ and $\rho$, readers are referred to corresponding papers and Theorem \ref{theorem:final_convergence} for details. }
    \begin{threeparttable}[t]
    \begin{tabular}{cccccc}
        \hline
        { Algorithm} & Condition & Design & Step sizes (upper, lower) & Rate \\\hline
        \begin{tabular}{c}
             BA  \\ (\cite{ghadimi2018approximation})
        \end{tabular} & \begin{tabular}{c}
             deterministic  \\ SC-SC
        \end{tabular} & double  & $\mathcal{O}(1)$, $\mathcal{O}(1)$ & $\mathcal{O}(\rho_1^{k})$ \tnote{(a)}\\
        \begin{tabular}{c}
             stocBIO  \\ (\cite{ji2021bilevel})
        \end{tabular} & \begin{tabular}{c}
             stochastic  \\ NC-SC
        \end{tabular}& double & $\mathcal{O}(1)$, $\mathcal{O}(1)$ & $\mathcal{O}(k^{-1})$ \tnote{(b)}\\
        \begin{tabular}{c}
             STABLE  \\ (\cite{chen2022single})
        \end{tabular} & \begin{tabular}{c}
             stochastic  \\
             SC-SC
        \end{tabular} & single & $\mathcal{O}(k^{-1/2}) $, $\mathcal{O}(k^{-1/2}) $ & $\mathcal{O}(k^{-1})$ \\
        \begin{tabular}{c}
             TTSA  \\ (\cite{hong2023two})
        \end{tabular} & \begin{tabular}{c}
             stochastic\\ SC-SC
        \end{tabular} & single & $\mathcal{O}(k^{-1})$, $\mathcal{O}(k^{-2/3})$ & $\mathcal{O}(k^{-2/3})$ \\\hline
        \begin{tabular}{c}
             \textbf{Proposed method}  \\ (\textbf{Theorem} \ref{theorem:final_convergence})
        \end{tabular} & \begin{tabular}{c}
             \textbf{deterministic}  \\ \textbf{SC-SC}
        \end{tabular} & \bf{single} & $\bm{\mathcal{O}(1)}$, $\bm{\mathcal{O}(1)}$ & $\bm{\mathcal{O}(\rho^{k})}$ \\\hline
    \end{tabular}
    \begin{tablenotes}
        \item[(a)] {The linear rate solely considers the update iterations of the upper-level, overlooking the gradually increasing number of lower-level gradient evaluations that might be involved in each evaluation of the upper-level gradient. }
        \item[(b)] In the nonconvex setting, {the} analysis establishes convergence towards the stationary point.
    \end{tablenotes}
    \end{threeparttable}
    \label{tab:convergence_rate}
\end{table}

To tackle the problem \eqref{equ:BLO-problem}, a double-loop iterative algorithm is first adopted in \cite{ghadimi2018approximation}, {where} {the} inner loop {searches} the minimizer of $g(\omega, \cdot)$ given fixed $\omega$, and the outer loop minimizes the upper-level objective function. Variants of double-loop algorithms can be found in \cite{ji2021bilevel, arbel2021amortized} and their applications in \cite{shou2020reward, franceschi2017forward}. However, this approach is considered as {computationally} inefficient since {the lower-level solution becomes outdated with every update of $\omega$}, requiring repeated recalculations due to the shifting lower-level problem.
 \par

In contrast, single-loop algorithms (\cite{li2022fully, chen2022single}) update $\omega$ and $v$ concurrently. Since the gradient of the upper-level objective function ideally depends on the exact solution of the lower-level problem, {these algorithms} need to replace the exact gradient with an estimate that incorporates the current iterate $v_k$ of the lower-level variable $v$ to break the dependency. 
Compared to double-loop counterparts, single-loop algorithms are more computationally efficient and simpler to implement. Nonetheless, establishing their convergence and convergence rate is more challenging due to the use of an approximated upper-level gradient. Numerous studies have explored the convergence of single-loop algorithms. For instance, \cite{ji2021bilevel} proves a sublinear convergence rate $\mathcal{O}(k^{-1})$ for the implicit differentiation (AID) and iterative differentiation (ITD) methods, which computes the approximated gradient using either a single update or iterative updates per iteration. The papers \cite{chen2021closing, khanduri2021near} further incorporate stochastic approximation techniques. Due to the computational burden of {calculating} the inverse of the Hessian matrix in gradient approximation, \cite{li2022fully} proposes a Hessian inverse-free algorithm. The paper \cite{liu2022inducing} further applies the algorithm to incentive design problem and proves a sublinear convergence rate. The paper \cite{hong2023two} analyzes the convergence rate under the SC-SC setting. They propose a two-timescale stochastic approximation (TTSA) algorithm, where the upper- and lower-level step sizes decay at different rates, and demonstrate its convergence rate of $\mathcal{O}(k^{-2/3})$. However, they assume a bounded upper-level gradient which may conflict with the strong convexity condition on the upper-level function. Additionally, \cite{liang2023lower} derives an optimistic complexity bound for linear convergence under SC-SC, meaning that if the algorithm can achieve linear convergence, the rate cannot exceed its established limit. Their work also establishes linear convergence for a double-loop algorithm, wherein the update frequency of the inner loop is adaptively determined based on the targeted error of the loss function. To summarize, various single-loop algorithms have been developed and analyzed under different settings (such as convex-strongly-convex or stochastic settings). \par

In this paper, robust control techniques are employed to establish linear convergence of the order $\mathcal{O}(\rho^k)$, for some $\rho\in(0,1)$ specified in Theorem \ref{theorem:final_convergence}. We model the update process of the single-loop method as a dynamical system, with the gradients as inputs and the upper- and lower-level variables as system states. Inspired by the approach in \cite{hu2017control}, we apply a small-gain theorem to {examine} the input-output stability of a modified rendition of the modeled dynamical system. Several studies, including \cite{hu2017control, hu2017dissipativity,lessard2016analysis}, also {utilize} robust control {techniques} (such as integral quadratic constraints and the small-gain theorem) to derive convergence for optimization algorithms. However, these works primarily focus on single-level {optimization}, where the convergence result is {comparatively} more well-known to the community. \par


\textbf{Main contributions}: First, we establish a linear convergence rate for a single-loop algorithm under the SC-SC condition, as formalized in Theorem \ref{theorem:final_convergence}. To the best of our knowledge, this presents the first result that demonstrates linear convergence for the single-loop algorithm. Prior studies (\cite{hong2023two,chen2022single}) predominantly achieve sublinear rates in stochastic gradient setting. 
These studies typically rely on the assumption of a bounded upper-level gradient (i.e., $\|\nabla f_*(\omega)\|$ is bounded; see Remark~\ref{remark1} for a detailed discussion.). This restrictive assumption prevents the derivation of linear convergence rates, even in deterministic setting.
Second, in contrast to the bounded upper-level gradient assumption, we adopt a gradient Lipschitz condition (Assumption \ref{assumption_upper}.\ref{assumption_upper_smooth}). It aligns the theoretical derivations more closely with practical scenarios, as the bounded upper-level gradient assumption may conflict with the SC-SC condition.

\textbf{Notation}: Unless otherwise specified, lowercase letters (e.g., $\omega, v$) denote vectors, while uppercase letters ($M$) denote matrices. The gradient of {the twice continuously differentiable} function $g(\omega, \cdot)$ with respect to $v$ is denoted by $\nabla_v g(\omega, v)$, and the Hessian matrix with respect to $v$ is denoted by $\nabla_{vv}^2g(\omega, v)$. We denote $\nabla_{\omega v}^2g(\omega, v):=\frac{\partial}{\partial \omega} \frac{\partial}{\partial v} g(\omega, v)$.
 Let $\Vert v\Vert_2$ denote the standard Euclidean norm of a vector $v$. 
Let $\ell_{2e}^p $ denote the set of all one-sided sequences $\{x_0,x_1,\ldots\}$ with each $x_k\in\mathbb R^p$, and let $\ell_2^p\subseteq \ell_{2e}^p$ denote the set of square-summable sequences, i.e., those with  bounded $\ell_2$ norm $\|x\|:=\sum_{k=0}^{\infty}\|x_k\|_2^2$. The superscript $p$ will be skipped if it is clear from the context.  The gain of a causal  bounded operator $K:\ell_{2e}\to \ell_{2e}$ is induced by $\ell_{2}$ signals, defined as 
$\|K\|=\sup_{x\in \ell _2,x\not=0}\frac{\|Kx\|}{\|x\|}$. 
In particular, when $K$ is a linear time-invariant system, the $\mathcal H_\infty$ norm of the transfer function $\hat K(z)$ of $K$ is 
defined as $\|\hat K(z)\|_{\infty}:={\rm ess}\sup_{\omega \in [-\pi,\pi]} \sigma_{\max}(\hat K({\mathrm{e}^{\mathrm{\mathbf{j}}\omega}}))$, which evaluates the highest singular value of $\hat K(\mathrm{e}^{\mathrm{\mathbf{j}} \omega})$ over the frequency range $[-\pi,\pi]$, and has
$\|\hat K(z)\|_{\infty}=\|K\|$.

\section{Preliminary}
\label{sec:preliminary}
\subsection{The Single-loop Algorithm}
\label{subsec:single_loop_algorithm}
We are concerned about the single-loop algorithm as studied in \cite{hong2023two, chen2022single}, which updates $\omega, v$ simultaneously. 
The underlying {intuition} is as follows.
When $g(\omega, \cdot)$ is strongly-convex, this function has a unique minimizer for any $\omega$, denoted as $v_*(\omega)$. The upper-level objective function can then be written as $f_*(\omega):=f(\omega, v_*(\omega))$. Therefore, its gradient can be computed as
\begin{equation*}
    \nabla_\omega f_*(\omega) = \nabla_\omega f(\omega, v_*(\omega)) +   \nabla_\omega v_*(\omega)\nabla_v f(\omega, v_*(\omega)).
\end{equation*}
Note that the gradient typically requires access to $v_*(\omega)$, which is in general not available unless the lower-level problem admits a closed-form solution. {To evaluate $\nabla_\omega v_*(\omega)$,} by differentiating both sides of the first-order optimality condition $\nabla_v g(\omega, v_*(\omega))=0$ with respect to $\omega$, we have 
\begin{equation*}
    \nabla_\omega v_*(\omega)\nabla_{vv}^2 g(\omega, v_*(\omega)) + \nabla_{\omega v}^2 g(\omega, v_*(\omega)) = 0.
\end{equation*}
With the strong convexity assumption,  $g(\omega, \cdot)$ admits an invertible Hessian {matrix}. In virtue of the implicit function theorem (\cite{ghadimi2018approximation}), we have 
\begin{equation*}
    \nabla_\omega f_*(\omega)=
    \nabla_\omega f(\omega, v_*(\omega)) -\nabla_{\omega v}^2 g(\omega, v_*(\omega))[\nabla_{vv}^2 g(\omega, v_*(\omega))]^{-1} \nabla_v f(\omega, v_*(\omega)).
\end{equation*}
In order to devise an iterative update, we replace $v_*(w)$ in the above formula with an arbitrary $v$ and define the so-called approximated gradient as follows:
\begin{equation*}
    \tilde{\nabla}f(\omega, v) := \nabla_\omega f(\omega, v) - \nabla_{\omega v}^2g(\omega, v)[\nabla_{vv}^2g(\omega, v)]^{-1} \nabla_v f(\omega, v).
\end{equation*}
Finally, the update rule of $\omega$ and $v$ can be formalized as 
\begin{equation}\label{equ:update_dynamics}
    \begin{aligned}
        \omega_{k+1} =  \omega_k - \alpha \Tilde{\nabla}f(\omega_k, v_k),\\
        v_{k+1} = v_k - \beta \nabla_v g(\omega_k, v_k),
    \end{aligned}
\end{equation}
where $\alpha, \beta$ are step sizes. \par


\subsection{Assumptions}
\label{subsec:assumptions}
In this subsection, {we characterize the assumptions necessary for the convergence rate analysis of the algorithm (\ref{equ:update_dynamics}).}
\begin{assumption}\label{assumption_lower}
    The lower-level objective function satisfies
    \begin{enumerate}
        \item \label{assumption_lower_strongly} Function $g(\omega, \cdot)$ is $\mu_g$-strongly convex, i.e., $\langle \nabla_v g(\omega, v), v'-v \rangle \leq$ $ -\frac{\mu_g}{2}\Vert v-v' \Vert_2^2$ holds for any $\omega, v, v'$;
        \item \label{assumption_lower_smooth} Function $g(\omega, \cdot)$ is $L_g$-smooth, i.e., $\Vert \nabla_v g(\omega, v) - \nabla_v g(\omega, v') \Vert_2 \leq L_g \Vert v-v' \Vert_2$ holds for all $\omega, v, v'$.
    \end{enumerate}
\end{assumption}
\begin{assumption}\label{assumption_upper}
The upper-level objective function satisfies
\begin{enumerate}
    \item \label{assumption_upper_strongly} Function $f_*(\cdot)$ is $\mu_f$-strongly convex, i.e., $\langle \nabla_\omega f_*(\omega), \omega'-\omega \rangle \leq -\frac{\mu_f}{2}\Vert \omega - \omega' \Vert_2^2$ holds for all $\omega, \omega'$;
    \item \label{assumption_upper_smooth} {The approximated gradient $ \tilde{\nabla}f(\omega, v)$ is $H_\omega$-Lipschitz with respect to $\omega$, and $H_v$-Lipschitz with respect to $v$, i.e., for all $\omega,\omega' v, v'$,}
    \begin{equation*}
        \begin{aligned}
            \Vert \tilde{\nabla}f(\omega, v) - \tilde{\nabla}f(\omega',v)\Vert&\leq H_\omega \Vert \omega - \omega' \Vert_2, \\
            \Vert \tilde{\nabla}f(\omega, v) - \tilde{\nabla}f(\omega, v') \Vert & \leq H_v \Vert v- v' \Vert_2.
        \end{aligned}
    \end{equation*}
\end{enumerate}
\end{assumption}
Assumptions \ref{assumption_lower} and \ref{assumption_upper}.\ref{assumption_upper_strongly} are 
customary for analysis of bilevel optimization with SC-SC setting, used in 
\cite{liu2022inducing,hong2023two,ghadimi2018approximation}.
The $H_\omega$-Lipschitz property in Assumption~\ref{assumption_upper}.\ref{assumption_upper_smooth} is used in \cite{liu2022inducing} and can be derived from Assumptions 2.1 and 2.2  in \cite{hong2023two} or from the identical assumptions {(Assumption 1 and 2)} in~\cite{ghadimi2018approximation}. 
\begin{remark}\label{remark1}
Prior studies (\cite{chen2022single, hong2023two}) commonly assume a uniformly bounded upper-level gradient $\nabla_\omega f_*(\omega)$. However, such an assumption is overly restrictive and may potentially conflict with the strong convexity of $\nabla_\omega f_*(\omega)$ (Assumption~\ref{assumption_upper}.\ref{assumption_upper_strongly}).
In contrast, we substitute this assumption with $H_v$-Lipschitz property in
Assumption \ref{assumption_upper}.\ref{assumption_upper_smooth} in the derivation of our main result, which 
is less restrictive in the
sense that it better aligns with real-world scenarios. 
\end{remark}

We further assume that the Hessian $\nabla_{\omega v}^2g(\omega, v)$ is bounded.
\begin{assumption}\label{assumption_continuous_v*}
    There exists $H>0$ such that, for all $\omega$, $v$, $\Vert \nabla_{\omega v}^2g(\omega, v) \Vert_2 < H$.
\end{assumption}
Given Assumption \ref{assumption_lower}.\ref{assumption_lower_strongly} that $g(\omega, \cdot)$ is $\mu_g$-strongly convex, we have $\Vert [ \nabla_{vv}^2g(\omega, v) ]^{-1} \Vert_2\leq 1/\mu_g$. Consequently, the lower-level solution $v_*(\cdot)$ is Lipschitz continuous, proved by following lemma. 
\begin{lemma}[Lemma B.3 in \cite{liu2022inducing}]
    \label{lemma:v*_lipschitz} For problem \eqref{equ:BLO-problem}, under Assumptions \ref{assumption_lower}.\ref{assumption_lower_strongly} and \ref{assumption_continuous_v*}, we have
$ \Vert v_*(\omega) - v_*(\omega') \Vert_2 \leq \frac{H}{\mu_g} \Vert \omega-\omega' \Vert_2$.
\end{lemma}

\section{Main Results}
\label{sec:main_results}
In this section, we derive the linear convergence rate {of the single-loop algorithm~\eqref{equ:update_dynamics} for solving the bilevel optimization under the SC-SC condition}. The update rule \eqref{equ:update_dynamics} is modeled as a dynamical system in Section \ref{subsec:dynamic_system}, with the nonlinear components (i.e., the gradients) as control inputs, and the linear components as a linear state space model.  We then compute the gain of the {nonlinear component} in Section \ref{subsec:gain_nonlinear}. In Section \ref{subsec:final_results}, by {optimizing} the gain of the linear {part}, we establish a set of feasible step sizes that enable the single-loop algorithm to converge linearly, along with a characterization of the convergence rate. 

\subsection{ The Dynamical System}
\label{subsec:dynamic_system}
We begin by modeling the iterative update \eqref{equ:update_dynamics} of variables as a dynamical system. Since
the equilibrium of~\eqref{equ:update_dynamics} is in general not at the origin, for the convenience of analysis, we introduce a coordinate shift to reshape the variable update dynamics. Specifically, we define a new state $x_k\in\mathbb R^{m+n}$ as
\begin{equation}\label{equ:state_def}
x_k:=\begin{bmatrix}
    x_{1,k} \\ x_{2,k}
\end{bmatrix}:=\begin{bmatrix}
\omega_k - \omega^*\\v_k-v_*(\omega_k)\end{bmatrix},
\end{equation}
where $\omega^*$ is the minimizer of $f_*(\cdot)$. Sequently, by introducing a nonlinear function 
$\phi(x_k)$, which is component-wise defined as  
\begin{equation}\label{equ:control_inputs}
\begin{aligned}
    \phi_{1}(x_k):=&\tilde{\nabla}f(x_{1,k}+\omega^*, x_{2,k}+v_*(\omega_k)),\\
    \phi_{2}(x_k):=&\nabla_v g(x_{1,k}+\omega^*, x_{2,k}+v_*(\omega_k)) +\frac{1}{\beta} ( v_*(x_{1,k}+\omega^*-\alpha \phi_{1}(x_k)) - v_*(x_{1,k}+\omega^*) ),
    \end{aligned}
\end{equation}
\begin{figure}[t]
    \centering
    \includegraphics[width=0.85\linewidth]{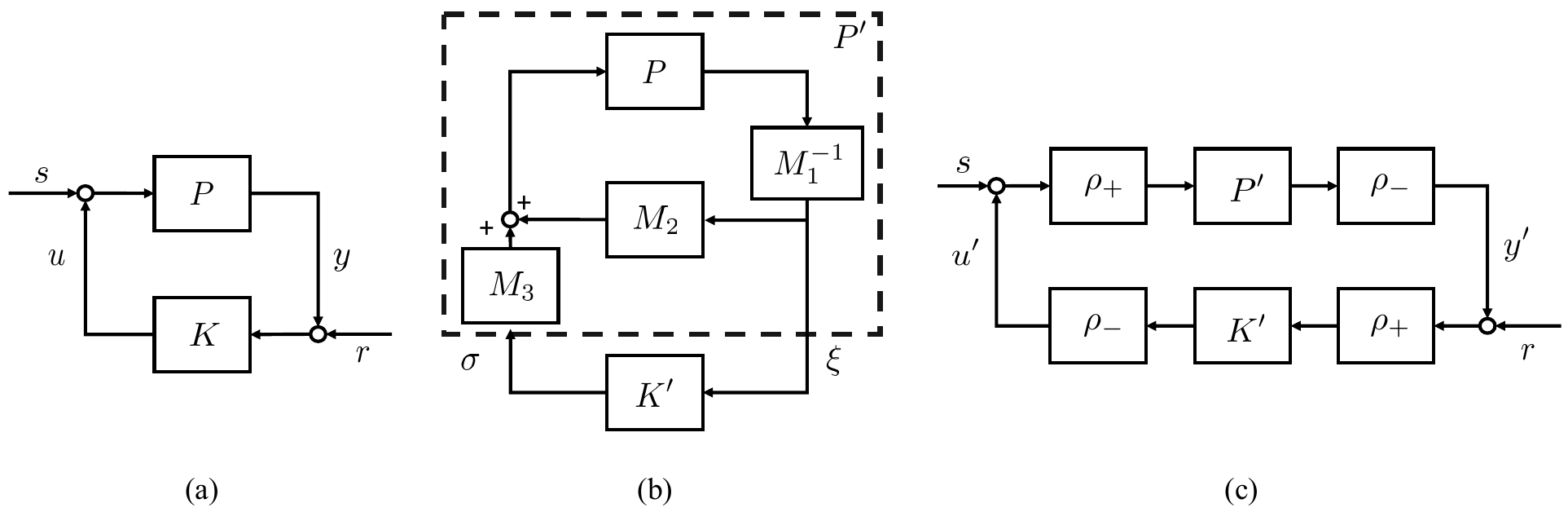}
    \vspace{-8pt}
    \caption{(a) The feedback interconnection of $P$ and $K$ with external inputs $s$ and $r$. (b) {The diagram of 
    the system after  implementing the linear transformation $M$, as specified by~\eqref{eqn:linear_transformation_M}, onto the system depicted in (a). The inputs $s$ and $r$ are absent.} (c) The feedback interconnection modified by introducing operators $\rho_-:=\rho^{-k}, \rho_+:=\rho^{k}$ for some given $\rho\in(0,1)$.   }
    \label{fig:enter-label}
    \vspace{-20pt}
\end{figure}
we obtain the following dynamics in a compact form for~\eqref{equ:update_dynamics}:
\begin{equation}\label{equ:update_process_modified}
        x_{k+1}  = x_k -  {\rm diag}(\alpha I_m, \beta I_n)\phi(x_k).
\end{equation}
Following  the conventional representation of feedback systems in control theory, it can be further cast as the interconnection $[P,K]$ of a linear system (in state-space form) {with} $P$:
\begin{equation}\label{eqn:system_P}
\begin{aligned}
x_{k+1}&=Ax_k+B u_k,\\
y_k&=Cx_k+D u_k,
\end{aligned}
\end{equation}
where $A=I_{m+n}$, $B={\rm diag}(\alpha I_m, \beta I_n)$, $C=I_{m+n}$ and $D=0$ ($x_k$, $y_k$ will be used interchangeably since they are identical.), and a memoryless nonlinear component {$K$: $u_k=\phi(y_k)$}. {See} Figure~\ref{fig:enter-label}(a) for  the system diagram 
with external inputs at rest, i.e., $s,r=0$.\par

We introduce different stability definitions of a dynamical system (\cite{dullerud2013course}) with setup specified in~\eqref{eqn:system_P}.
\begin{definition}[Internal Stability]\label{def:exp_stable}
    The system  in Figure \ref{fig:enter-label}(a)  is said to be internally stable if 
    it is well-posed\footnote{The system in Figure \ref{fig:enter-label}(a)  is well posed if unique solutions exist for  $x_{k}$,  $u_{k}$ and $y_k$, for all initial conditions  $x_{0}$ and all regular inputs $s_0$ and $r_0$.}, and for every initial conditions
    $x_{0}$ of $P$, $x_{k}\to 0$ as $k\to \infty$ with $s,r=0$.
    It is said to be exponentially stable with rate $\rho$ if there exist constants $c>0$ and $\rho\in(0,1)$ such that 
    $\|x_k\|_2\leq c\rho^k\|x_0\|_2$ with any initial state $x_0$ and $s,r=0$.
    
\end{definition}
\begin{definition}[Bounded-input-bounded-output (BIBO) Stability]\label{def:bibo_stable}
    The system  in Figure \ref{fig:enter-label}(a)  is said to be bounded-input-bounded-output stable if the closed-loop map $\begin{bmatrix}r\\s\end{bmatrix}\to \begin{bmatrix}u\\y\end{bmatrix}$ is a bounded, casual operator. 
\end{definition}
Our primary approach to certify the linear convergence of \eqref{equ:update_dynamics} is to establish a connection with the BIBO stability of a modified system. To achieve this, we introduce the operators $\rho_+$ and $\rho_-$ as two time-varying multipliers $\rho^k$ and $\rho^{-k}$, where $\rho\in(0,1)$ is a constant. By breaking the {interconnection} of $[P,K]$ and incorporating $\rho_+$ and $\rho_-$ at the break point, we derive a modified system resembling the configuration in Figure~\ref{fig:enter-label}(c). Since for a signal $u$ 
with $z$-transform $\hat u(z)$, the $z$-transforms of $\rho_+ u$ and 
$\rho_- u$ are $\hat u(z/\rho)$ and $\hat u(\rho z)$, {respectively,} the operator $\rho_-\circ P \circ \rho_+$ has $z$-transform $\hat P(\rho z)$. The following lemma formally articulates the connection.  
\begin{lemma}[Proposition 5 in \cite{boczar2015exponential}]
\label{lemma:CDC15_system_stbility_conversion}
     Suppose the system $P$ has a minimal realization. If the interconnection in Figure \ref{fig:enter-label}(c) is BIBO stable, then the system in Figure \ref{fig:enter-label}(a) is exponentially stable  with rate $\rho$. 
\end{lemma}
{The small-gain theorem is used to certify the BIBO stability.}
\begin{lemma}[Small-gain theorem in \cite{desoer2009feedback}]
    {Suppose $[P,K]$ shown in Figure \ref{fig:enter-label}(a) is well-posed, and $P$ and $K$ are bounded causal operators. If $\|K\| \|P \|< 1$, then $[P,K]$ is BIBO stable. }
\end{lemma}

The roadmap for overall proof unfolds as follows. First we relocate the nonlinearity 
$K$ into a sector-bounded region by a linear transformation applied to 
$x_k,~u_k$ of the system 
$[P,K]$, thus determining the gain of the nonlinear mapping. Subsequently, we build the modified system as depicted in Figure \ref{fig:enter-label}(c) by incorporating the $\rho_+$ and $\rho_-$ operators. Finally, we separately compute the gain to analyze the BIBO stability of the system using the small-gain theorem. 


\subsection{Gain of Nonlinear Component}
\label{subsec:gain_nonlinear}
In this subsection, we compute the gain of the nonlinear operator $K$. We first derive the following lemmas that characterize the inequalities governing the state \eqref{equ:state_def} and control input \eqref{equ:control_inputs}.

\begin{lemma}
\label{lemma:upper_iqc}
    Under Assumption \ref{assumption_upper}, for any $k$, we have
    \begin{gather}
    \label{equ:upper_iqc_1}  
         \langle \tilde{\nabla}f(\omega_k, v_k), \omega_k - \omega^* \rangle \geq \frac{3\mu_f}{8}\Vert \omega_k-\omega^*   \Vert_2^2 - \frac{2H_v^2}{\mu_f} \Vert v_k - v_*(\omega_k) \Vert_2^2,\\
    \label{equ:upper_iqc_2}
            \Vert \tilde{\nabla}f(\omega_k, v_k) \Vert_2^2 \leq 2(H_\omega^2+2\frac{H_v^2H^2}{\mu_g^2}) \Vert \omega_k -\omega^*\Vert_2^2 + 4H_v^2 \Vert v_k-v_*(\omega_k)  \Vert_2^2.
    \end{gather}
\end{lemma}
\begin{proof}
 Routine calculation yields that
\begin{equation*}
    \begin{aligned}
         -\langle \tilde{\nabla}f(\omega_k, v_k), \omega_k - \omega^*\rangle
        = & \langle \nabla f_*(\omega_k), \omega^* - \omega_k \rangle+ \langle \tilde{\nabla}f(\omega_k, v_k) - \nabla f_*(\omega_k), \omega^* - \omega_k \rangle \\ 
        \leq & -\frac{\mu_f}{2}\Vert \omega_k -\omega^*    \Vert_2^2 + \frac{2}{\mu_f} \Vert \tilde{\nabla}f(\omega_k, v_k) - \nabla f_*(\omega_k)\Vert_2^2 + \frac{\mu_f}{8}\Vert \omega_k-\omega^*   \Vert_2^2 \\
        \leq & -\frac{3\mu_f}{8}\Vert \omega_k-\omega^* \Vert_2^2 + \frac{2H_v^2}{\mu_f} \Vert v_k- v_*(\omega_k)  \Vert_2^2,
    \end{aligned}
\end{equation*}
where the first inequality is due to Assumption \ref{assumption_upper}.\ref{assumption_upper_strongly} and Cauchy-Schwartz inequality, and the second inequality is due to Assumption \ref{assumption_upper}.\ref{assumption_upper_smooth}. As for \eqref{equ:upper_iqc_2}, we have
\begin{equation*}
    \begin{aligned}
        \Vert \tilde{\nabla}f(\omega_k, v_k) \Vert_2^2 \leq & 2 \Vert \tilde{\nabla}f(\omega_k, v_k) - \tilde{\nabla}f(\omega^*, v_k) \Vert_2^2 + 4\Vert \tilde{\nabla}f(\omega^*, v_k) -\tilde{\nabla}f(\omega^*, v_*(\omega_k)) \Vert_2^2 \\
        & + 4\Vert \tilde{\nabla}f(\omega^*, v_*(\omega_k)) - \nabla f_*(\omega^*)) \Vert_2^2\\
        \leq & 2(H_\omega^2+2\frac{H_v^2 H^2}{\mu_g^2}) \Vert \omega_k-\omega^* \Vert_2^2 + 4H_v^2 \Vert v_k - v_*(\omega_k) \Vert_2^2,
    \end{aligned}
\end{equation*}
where the second inequality is based on Assumption \ref{assumption_upper}.\ref{assumption_upper_smooth} and Lemma \ref{lemma:v*_lipschitz}.
\end{proof}

\begin{lemma}
    \label{lemma:lower_iqc}
    Under Assumption \ref{assumption_lower}, for any $k$, we have
    \begin{equation}
    \label{equ:lower_iqc_1}
            \langle \nabla_v g(\omega_k, v_k) + \frac{1}{\beta} (v_*(\omega_{k+1}) - v_*(\omega_k)), v_k - v_*(\omega_k) \rangle \geq  \frac{3\mu_g}{8} \Vert v_k-v_*(\omega_k) \Vert_2^2 - \frac{2\alpha^2H^2}{\beta^2 \mu_g^3} \Vert \tilde{\nabla}f(\omega_k, v_k) \Vert_2^2,
    \end{equation}
    \begin{equation}        
    \label{equ:lower_iqc_2}
             \Vert \nabla_v g(\omega_k, v_k) + \frac{1}{\beta} (v_*(\omega_{k+1}) -  v_*(\omega_k)) \Vert_2^2 \leq  2L_g^2 \Vert v_k-v_*(\omega_k) \Vert_2^2 + \frac{2\alpha^2 H^2}{\beta^2\mu_g^2}\Vert \tilde{\nabla}f(\omega_k, v_k) \Vert_2^2.
    \end{equation}
\end{lemma}
\begin{proof} Routine calculation yields that
\begin{dmath}
    \begin{aligned}
        & -\langle \nabla_v g(\omega_k, v_k) + \frac{1}{\beta}v_*(\omega_{k+1}) -\frac{1}{\beta} v_*(\omega_k) , v_k - v_*(\omega_k) \rangle \\
        \leq & -\frac{\mu_g}{2}\Vert v_k-v_*(\omega_k) \Vert_2^2 + \frac{2}{\beta^2\mu_g} \Vert v_*(\omega_{k+1})-v_*(\omega_k) \Vert_2^2 + \frac{\mu_g}{8}\Vert v_k-v_*(\omega_k) \Vert_2^2 \\
        \leq & -\frac{3\mu_g}{8}\Vert v_k-v_*(\omega_k) \Vert_2^2 + \frac{2\alpha^2 H^2}{\beta^2 \mu_g^3} \Vert \tilde{\nabla}f(\omega_k, v_k) \Vert_2^2,
    \end{aligned}
    \end{dmath}
where the first inequality is due to Assumption \ref{assumption_lower}.\ref{assumption_lower_strongly}, and the third {one} is due to Lemma \ref{lemma:v*_lipschitz} as well as (\ref{equ:update_dynamics}). As for \eqref{equ:lower_iqc_2}, we have \begin{equation*}
    \begin{aligned}
        & \Vert \nabla_v g(\omega_k, v_k) + \frac{1}{\beta}v_*(\omega_{k+1})-\frac{1}{\beta} v_*(\omega_k) \Vert_2^2 \\
        \leq & 2\Vert \nabla_v g(\omega_k, v_k) - \nabla_v g(\omega_k, v_*(\omega_k)) \Vert_2^2 + \frac{2}{\beta^2}\Vert v_*(\omega_{k+1}) - v_*(\omega_k) \Vert_2^2 \\
        \leq & 2L_g^2\Vert v_k - v_*(\omega_k) \Vert_2^2 + \frac{2\alpha^2H^2}{\beta^2\mu_g^2} \Vert \tilde{\nabla}f(\omega_k, v_k) \Vert_2^2,
    \end{aligned}
\end{equation*}
where the second {inequality is} due to Assumption \ref{assumption_lower}.\ref{assumption_lower_smooth}.
\end{proof} Next we adopt a linear transformation to $x_k$ and $u_k$ and further derive the gain of a transformed nonlinear mapping on top of Lemmas \ref{lemma:upper_iqc} and \ref{lemma:lower_iqc}. 

\begin{lemma}
\label{lemma:nonlinear_gain}
    If there exist positive scalars $\lambda_1, \lambda_2, \lambda_3, \lambda_4$, $\alpha$, and $\beta$ satisfying following conditions
\begin{align}
            \frac{3\lambda_1^2}{4\lambda_3} &> \frac{3\mu_f}{8}\lambda_1 - 2(H_\omega^2+2\frac{H_v^2H^2}{\mu_g^2}) \lambda_3 \label{equ:lemma_3_4_assume_1},\\
            \frac{\lambda_2^2}{4\lambda_4} & > -\frac{2H_v^2}{\mu_f}\lambda_1 + \frac{3\mu_g}{8}\lambda_2 - 4H_v^2\lambda_3 - 2L_g^2\lambda_4 \label{equ:lemma_3_4_assume_2},  \\
            \frac{2}{3}\lambda_3 &\geq \frac{2\alpha^2H^2}{\beta^2\mu_g^3}\lambda_2 + \frac{2\alpha^2H^2}{\beta^2\mu_g^2}\lambda_4\label{equ:lemma_3_4_assume_3},
    \end{align}
    then there exists a linear transformation $M\in\mathbb R^{(m+n)\times (m+n)}$,
    defining 
$\begin{bmatrix}
             \xi_k  \\
             \sigma_k 
        \end{bmatrix}=  M\begin{bmatrix}
             x_k  \\
             u_k 
        \end{bmatrix}$,
    such that the transformed state $\xi_k$ and input $\sigma_k$ satisfy $(\sigma_k-\xi_k)^\top (\sigma_k+\xi_k) \leq 0$.
\end{lemma}
\begin{proof}
From \eqref{equ:state_def} and \eqref{equ:control_inputs}, inequalities (\ref{equ:upper_iqc_1}),(\ref{equ:lower_iqc_1}),(\ref{equ:upper_iqc_2}),(\ref{equ:lower_iqc_2}) are equivalent to the following:  
$-\langle \phi_1(x_k), x_{1,k} \rangle \leq -\frac{3\mu_f}{8}\Vert x_{1,k} \Vert_2^2 + \frac{2H_v^2}{\mu_f} \Vert x_{2,k} \Vert_2^2$, $-\langle \phi_2(x_k), x_{2,k} \rangle \leq -\frac{3\mu_g}{8}\Vert x_{2,k} \Vert_2^2 + \frac{2\alpha^2H^2}{\beta^2\mu_g^3}\Vert \phi_1(x_k) \Vert_2^2$, $\Vert \phi_1(x_k) \Vert_2^2 \leq 2(H_\omega^2 +\frac{2H_v^2H^2}{\mu_g^2})\Vert x_{1,k} \Vert_2^2  + 4H_v^2 \Vert x_{2,k} \Vert_2^2$ and $ \Vert \phi_2(x_k) \Vert_2^2  \leq 2L_g^2 \Vert x_{2,k} \Vert_2^2 + \frac{2\alpha^2H^2}{\beta^2\mu_g^2}\Vert \phi_1(x_k) \Vert_2^2$.
Suppose there exists positive scalars $\lambda_1, \lambda_2, \lambda_3, \lambda_4, \alpha, \beta$ satisfying \eqref{equ:lemma_3_4_assume_1}, \eqref{equ:lemma_3_4_assume_2}, and \eqref{equ:lemma_3_4_assume_3}, we obtain following inequality in compact form by summing the above four inequalities using weight parameters $\lambda_1, \lambda_2, \lambda_3,\lambda_4$,
\begin{equation}
\label{equ:lemma_3_4_iqc1}
 \begin{bmatrix}
         x_k  \\
         u_k
    \end{bmatrix} ^\top\begin{bmatrix}
        a I_m & & -\frac{\lambda_1}{2}I_m & \\
         & bI_n & & -\frac{\lambda_2}{2}I_n \\
         -\frac{\lambda_1}{2}I_m & & \frac{\lambda_3}{3} I_m & \\
         & -\frac{\lambda_2}{2}I_n & & \lambda_4 I_n 
    \end{bmatrix}   \begin{bmatrix}
         x_k\\
         u_k
    \end{bmatrix} := \begin{bmatrix}
         x_k  \\
         u_k
    \end{bmatrix} ^\top  N_0\begin{bmatrix}
         x_k  \\
         u_k
    \end{bmatrix} \leq 0, 
\end{equation}
where $a:=\frac{3\mu_f}{8}\lambda_1 - 2(H_\omega^2 +2 \frac{H_v^2H^2}{\mu_g^2})\lambda_3$, $b:=-\frac{2H_v^2}{\mu_f}\lambda_1 + \frac{3\mu_g}{8}\lambda_2-4H_v^2\lambda_3-2L_g^2\lambda_4$.  Let \begin{equation}\label{eqn:linear_transformation_M}
M=\left[\begin{array}{ c c : c c}
         \frac{1}{\sqrt{\frac{3\lambda_1^2}{4\lambda_3}-a}}I_m& & &  \\
         & \frac{1}{\sqrt{\frac{\lambda_2^2}{4\lambda_4}-b}}I_n & & \\\hdashline
         \frac{3\lambda_1}{2\lambda_3}\frac{1}{\sqrt{\frac{3\lambda_1^2}{4\lambda_3}-a}}I_m & & \sqrt{\frac{3}{\lambda_3}} I_n & \\
         & \frac{\lambda_2}{2\lambda_4}\frac{1}{\sqrt{\frac{\lambda_2^2}{4\lambda_4}-b}}I_n & & \frac{1}{\sqrt{\lambda_4}}I_n\end{array}
    \right]:=\begin{bmatrix}M_1 &0\\M_2 &M_3\end{bmatrix}.
    \end{equation}  
    We can verify that  $M^\top N_0 M={\rm diag}(-I_{m+n},I_{m+n})$.  Substituting $\begin{bmatrix}
        x_k\\
         u_k
    \end{bmatrix}  =  M \begin{bmatrix}
        \xi_k \\
        \sigma_k
    \end{bmatrix}$ to (\ref{equ:lemma_3_4_iqc1}) leads to 
$\sigma_k^\top \sigma_k\leq \xi_k^\top \xi_k$, which completes the proof.
\end{proof}
We remark that the conditions~\eqref{equ:lemma_3_4_assume_1}, \eqref{equ:lemma_3_4_assume_2} and~\eqref{equ:lemma_3_4_assume_3} in Lemma~\ref{lemma:nonlinear_gain} are 
not overly restrictive for ensuring the existence of $M$. At least, as long as $\frac{\lambda_1}{\lambda_3}$ and $\frac{\lambda_2}{\lambda_4}$ are sufficiently large and $\frac{\alpha}{\beta}$ is sufficiently small (Indeed, a specific threshold for $\frac{\alpha}{\beta}$ is discussed in Theorem~\ref{theorem:final_convergence}.), these conditions are satisfied. 
From the nonlinear component $K$ and the linear transformation $M$ that converts $x_k,u_k$ from $\xi_k,\sigma_k$, an equivalent nonlinearity $K'$ that directly maps $\xi_k$ to $\sigma_k$ can be outlined. The formulation is expressed as $K':\xi_k\mapsto \sigma_k:=M_3^{-1}(\phi(M_1 \xi_k)-M_2\xi_k)$, where $M_1,~M_2,~M_3$ are specified in~\eqref{eqn:linear_transformation_M}. 
Moreover, the gain of $K'$ satisfies $\Vert K'\Vert=\sup_{\xi_k\neq0}\frac{\Vert\sigma_k\Vert}{\Vert \xi_k \Vert}\leq 1$, indicated by
Lemma~\ref{lemma:nonlinear_gain}. Finally, to pair with $K'$ to close the loop, we need to construct a new linear system $P'$ from 
$P$ in tandem; see Figure~\ref{fig:enter-label}(b) for the new interconnection between $P'$ and $K'$. 




\subsection{Putting Things Together}
\label{subsec:final_results}
To prove the final convergence of the overall single-loop algorithm, we need to construct a set of parameters $\lambda_1, \lambda_2, \lambda_3, \lambda_4$ and step sizes $\alpha, \beta$ which satisfy the conditions outlined in Lemma \ref{lemma:nonlinear_gain} to ensure that $\|K'\|$ is less than 1 as well as the $H_\infty$ norm of the transformed linear state space also less than 1. 
We construct such parameters with following theorem. 

\begin{theorem}\label{theorem:final_convergence}
When the step sizes in the single-loop algorithm (\ref{equ:update_dynamics}) satisfy
\begin{equation*}
    \alpha < \min\{ \frac{\mu_f}{8(H_\omega^2 + \frac{2H_v^2H^2}{\mu_g^2})},\text{  } \frac{1}{24\mu_f} \},\text{  } \beta < \min\{ \frac{\mu_g}{8L_g^2}, \frac{1}{4\mu_g} \}, \frac{\alpha}{\beta^2} < \frac{2\mu_f\mu_g^4}{81H_v^2 H^2},
\end{equation*}
 the variables $\omega_k, v_k$ generated by  (\ref{equ:update_dynamics}) converges linearly to $\omega^*, v_*(\omega^*)$ with rate $\rho^k$, where 
 \begin{equation*}
    \rho \leq \max\left\{ \sqrt{ 1-\frac{3\mu_f\alpha}{4}(1-\frac{8(H_\omega^2+2\frac{H_v^2H^2}{\mu_g^2})\alpha}{\mu_f})}, \sqrt{1-\frac{\mu_g\beta}{2}(1-4 L_g^2\beta) } \right\},
\end{equation*}
i.e., there exists constants $c_\omega, c_v>0$ such that $\|\omega_k-\omega^*\|_2\leq c_\omega\rho^k(\|w_0-\omega^*\|_2+\| v_0 - v_*(\omega_0) \|_2)$, and $\|v_k-v_*(\omega^*)\|\leq c_v \rho^k(\|w_0-\omega^*\|_2+\| v_0 - v_*(\omega_0) \|_2)$.
\end{theorem}
\begin{proof}
Given parameters $\lambda_1, \lambda_2, \lambda_3, \lambda_4, \alpha$ and $\beta$ satisfying conditions in Lemma \ref{lemma:nonlinear_gain}, a real matrix $M$ specified in \eqref{eqn:linear_transformation_M} exists. We therefore construct the following $P'$ to accomplish the new system interconnection $[P',K']$:
\begin{equation*}
    \xi_{k+1} = 
    \mathrm{diag}\left(
        (1-\frac{3\lambda_1}{2\lambda_3}\alpha)I_m, (1-\frac{\lambda_2}{2\lambda_4}\beta) I_n\right)
 \xi_k - \mathrm{diag}\left(
        \alpha \sqrt{\frac{9\lambda_1^2}{4\lambda_3^2} - \frac{3a}{\lambda_3}}I_m, 
         \beta \sqrt{ \frac{\lambda_2^2}{4\lambda_4^2} - \frac{b}{\lambda_4} } I_n\right)
    \sigma_k.
\end{equation*}
Note that its transfer matrix is diagonal. We  compute each term on the diagonal as follows: $
    \hat P'_{1}(z)=\frac{\alpha\frac{3\lambda_1}{2\lambda_3}\sqrt{ 1 - \frac{4\lambda_3a}{3\lambda_1^2} }}{z+\frac{3\lambda_1}{2\lambda_3}\alpha - 1 }$ and $\hat P'_{2}(z) = \frac{ \beta\frac{\lambda_2}{2\lambda_4}\sqrt{1 - \frac{4\lambda_4b}{\lambda_2^2}} }{z+\frac{\lambda_2}{2\lambda_4}\beta-1},
$ where $a$ and $b$ are defined in the proof of Lemma~\ref{lemma:nonlinear_gain}.
Next we consider the 
interconnection modified by introducing operators $\rho_-:=\rho^{-k}, \rho_+:=\rho^{k}$ for given $\rho\in(0,1)$.
The operator $\rho_-\circ P \circ \rho_+$ has $z$-transform $\hat P'_{\rho}(z)=
\hat P'(\rho z)$, {whose} first diagonal term has $\mathcal{H}_\infty$ norm
given by
\begin{equation*}
    \begin{aligned}
    \Vert \hat P'_{\rho,1}(z) \Vert_\infty &= \frac{ \sqrt{1-\frac{4\lambda_3 a}{3\lambda_1^2}} }{\frac{2\lambda_3}{3\lambda_1\alpha}(\rho-1)+1} \quad \text{ if } 0<\alpha\leq\frac{2\lambda_3}{3\lambda_1}, \\
    \Vert \hat P'_{\rho,1}(z) \Vert_\infty &= \frac{ \sqrt{1-\frac{4\lambda_3 a}{3\lambda_1^2}} }{ \frac{2\lambda_3}{3\lambda_1\alpha}(1+\rho)-1 } \quad \text{ if }  \alpha\geq \frac{2\lambda_3}{3\lambda_1}.
    \end{aligned}
\end{equation*} 
The smallest value of $\rho$ that makes 
this norm {less} than $1$, if possible, is consistently achieved when 
$\alpha=\frac{2\lambda_3}{3\lambda_1}$. A similar observation occurs to $\hat P'_{\rho,2}(z)$ with $\beta=\frac{2\lambda_4}{\lambda_2}$. \par

Therefore, if $\lambda_1, \lambda_2, \lambda_3, \lambda_4$ satisfy $a>0$, $b>0$ and~\eqref{equ:lemma_3_4_assume_1},~\eqref{equ:lemma_3_4_assume_2}, and~\eqref{equ:lemma_3_4_assume_3}, then
any $\rho>
\max\{\sqrt{1-\frac{4\lambda_3a}{3\lambda_1^2}}, \sqrt{1-\frac{4\lambda_4b}{\lambda_2^2}}\}$  will lead to $\|\hat P'_{\rho}(z)\|<1$, which further implies the linear convergence of $x_k$ in~\eqref{equ:update_dynamics} with a rate {no larger than} $\max\{\sqrt{1-\frac{4\lambda_3a}{3\lambda_1^2}}, \sqrt{1-\frac{4\lambda_4b}{\lambda_2^2}}\}$ due to Lemma~\ref{lemma:CDC15_system_stbility_conversion}. To complete the proof, it suffices to establish sufficient condition under which $\lambda_1$ to $\lambda_4$ are feasible to the above set of five inequalities.

First, by dividing both sides of $a>0$ and~\eqref{equ:lemma_3_4_assume_1} with $\lambda_1$ and substituting $\alpha = \frac{2\lambda_3}{3\lambda_1}$ in the inequalities, we obtain 
$\frac{3\mu_f}{8} - 3(H_\omega^2+2\frac{H_v^2H^2}{\mu_g^2})\alpha >0$ and 
    $\frac{3\mu_f}{8} - 3(H_\omega^2+2\frac{H_v^2H^2}{\mu_g^2})\alpha < \frac{1}{2\alpha}$, 
which hold when $\alpha<\min\{\frac{\mu_f}{8(H_\omega^2+\frac{2H_v^2H^2}{\mu_g^2})}, \frac{1}{24\mu_f}\}$. 
Similarly, by assigning $\frac{\lambda_1}{\lambda_2}=\frac{\mu_g}{8}/\left( \frac{2H_v^2}{\mu_f}+6H_v^2\alpha \right)$, we have 
\begin{equation*}\label{equ:thm_3_5_require_2_1}
    \begin{aligned}
        \frac{b}{\lambda_2}=& \left( -\frac{2H_v^2}{\mu_f}\lambda_1 + \frac{3\mu_g}{8}\lambda_2 -4H_v^2\lambda_3 - 2L_g^2\lambda_4 \right) / \lambda_2\\
        = & -\left(\frac{2H_v^2}{\mu_f}+6H_v^2\alpha\right)\frac{\lambda_1}{\lambda_2} + \frac{3\mu_g}{8} - L_g^2\beta
        =  \frac{\mu_g}{4} - L_g^2 \beta.
    \end{aligned}
\end{equation*}
The inequalities $b>0$
and~\eqref{equ:lemma_3_4_assume_2} hold when $\beta < \min\{ \frac{\mu_g}{8L_g^2}, \frac{1}{4\mu_g} \}$. 
Lastly, to find conditions such that  \eqref{equ:lemma_3_4_assume_3} holds, we divide the right hand side of \eqref{equ:lemma_3_4_assume_3} with $\lambda_3$ and obtain
\begin{equation*}
        \frac{2\alpha^2H^2}{\beta^2\mu_g^3}\frac{\lambda_2}{\lambda_3}+\frac{2\alpha^2H^2}{\beta^2\mu_g^2}\frac{\lambda_4}{\lambda_3} = \left( \frac{2}{\beta\mu_g}+1 \right) \frac{2\alpha H^2}{3\beta\mu_g^2} \frac{\lambda_2}{\lambda_1} 
        \leq  \frac{9}{4\beta\mu_g} \frac{2\alpha H^2}{3\beta\mu_g^2} \frac{18H_v^2}{\mu_f\mu_g} = \frac{\alpha}{\beta^2} \frac{27H_v^2H^2}{\mu_g^4\mu_f},
\end{equation*}
where the inequality is due to $\beta < \frac{1}{4\mu_g}$, $\alpha<\frac{1}{24\mu_f}$, and $\frac{\lambda_1}{\lambda_2}=\frac{\mu_g}{8}/\left( \frac{2H_v^2}{\mu_f}+6H_v^2\alpha \right)$. When $\frac{\alpha}{\beta^2}< \frac{2\mu_f\mu_g^4}{81H_v^2H^2}$, \eqref{equ:lemma_3_4_assume_3} holds.\par
Therefore, we have $\omega_k-\omega^*$ and $v_k-v_*(\omega_k)$ converge to zero linearly with rate $\rho^k$. Furthermore, with Lipschitz continuity of $v_*(\cdot)$ proved in Lemma \ref{lemma:v*_lipschitz}, {we have $\|v_k-v_*(\omega^*)\|_2\leq \|v_k-v_*(\omega_k)\|_2 + \| v_*(\omega_k)-v_*(\omega^*) \|_2\leq \| v_k-v_*(\omega_k) \|_2+H_*\| \omega_k-\omega^* \|_2$ also converge linearly.}
 \end{proof}
According to Theorem~\ref{theorem:final_convergence}, linear convergence of~\eqref{equ:update_dynamics} under the SC-SC condition necessitates sufficiently small step sizes $\alpha, \beta$, along with a modest ratio of $\frac{\alpha}{\beta^2}$. An intuitive rationale for a small ratio of $\frac{\alpha}{\beta^2}$ may be that it guarantees a more effective convergence of the inner problem relative to the optimization at the outer layer. 
This setting may facilitate pinpointing a refined point
at the lower level for conducting gradient approximations in the upper-level gradient descent.

\section{Conclusion}
In this paper, we consider a single-loop algorithm for solving SC-SC bilevel optimization problem. Our main contribution is {establishing} a linear convergence rate which is primarily faster compared with existing sublinear rates. The proof technique is mainly motivated by control theory {previously} used to prove the convergence of single-level optimization {problems}. By constructing a set of parameters which ensure the product of gains of nonlinear control component and linear state-space system component is less than 1, {we} prove the linear convergence of the algorithm via the small-gain theorem and provide the convergence rate.

\bibliography{main}

\begin{thebibliography}{25}
\providecommand{\natexlab}[1]{#1}
\providecommand{\url}[1]{\texttt{#1}}
\expandafter\ifx\csname urlstyle\endcsname\relax
  \providecommand{\doi}[1]{doi: #1}\else
  \providecommand{\doi}{doi: \begingroup \urlstyle{rm}\Url}\fi

\bibitem[Arbel and Mairal(2021)]{arbel2021amortized}
Michael Arbel and Julien Mairal.
\newblock Amortized implicit differentiation for stochastic bilevel
  optimization.
\newblock \emph{arXiv preprint arXiv:2111.14580}, 2021.

\bibitem[Boczar et~al.(2015)Boczar, Lessard, and Recht]{boczar2015exponential}
Ross Boczar, Laurent Lessard, and Benjamin Recht.
\newblock Exponential convergence bounds using integral quadratic constraints.
\newblock In \emph{2015 54th IEEE conference on decision and control (CDC)},
  pages 7516--7521. IEEE, 2015.

\bibitem[Chen et~al.(2021)Chen, Sun, and Yin]{chen2021closing}
Tianyi Chen, Yuejiao Sun, and Wotao Yin.
\newblock Closing the gap: Tighter analysis of alternating stochastic gradient
  methods for bilevel problems.
\newblock \emph{Advances in Neural Information Processing Systems},
  34:\penalty0 25294--25307, 2021.

\bibitem[Chen et~al.(2022)Chen, Sun, Xiao, and Yin]{chen2022single}
Tianyi Chen, Yuejiao Sun, Quan Xiao, and Wotao Yin.
\newblock A single-timescale method for stochastic bilevel optimization.
\newblock In \emph{International Conference on Artificial Intelligence and
  Statistics}, pages 2466--2488. PMLR, 2022.

\bibitem[Desoer and Vidyasagar(2009)]{desoer2009feedback}
Charles~A Desoer and Mathukumalli Vidyasagar.
\newblock \emph{Feedback systems: input-output properties}.
\newblock SIAM, 2009.

\bibitem[Dullerud and Paganini(2013)]{dullerud2013course}
Geir~E Dullerud and Fernando Paganini.
\newblock \emph{A course in robust control theory: a convex approach},
  volume~36.
\newblock Springer Science \& Business Media, 2013.

\bibitem[Franceschi et~al.(2017)Franceschi, Donini, Frasconi, and
  Pontil]{franceschi2017forward}
Luca Franceschi, Michele Donini, Paolo Frasconi, and Massimiliano Pontil.
\newblock Forward and reverse gradient-based hyperparameter optimization.
\newblock In \emph{International Conference on Machine Learning}, pages
  1165--1173. PMLR, 2017.

\bibitem[Franceschi et~al.(2018)Franceschi, Frasconi, Salzo, Grazzi, and
  Pontil]{franceschi2018bilevel}
Luca Franceschi, Paolo Frasconi, Saverio Salzo, Riccardo Grazzi, and
  Massimiliano Pontil.
\newblock Bilevel programming for hyperparameter optimization and
  meta-learning.
\newblock In \emph{International conference on machine learning}, pages
  1568--1577. PMLR, 2018.

\bibitem[Ghadimi and Wang(2018)]{ghadimi2018approximation}
Saeed Ghadimi and Mengdi Wang.
\newblock Approximation methods for bilevel programming.
\newblock \emph{arXiv preprint arXiv:1802.02246}, 2018.

\bibitem[Hong et~al.(2023)Hong, Wai, Wang, and Yang]{hong2023two}
Mingyi Hong, Hoi-To Wai, Zhaoran Wang, and Zhuoran Yang.
\newblock A two-timescale stochastic algorithm framework for bilevel
  optimization: Complexity analysis and application to actor-critic.
\newblock \emph{SIAM Journal on Optimization}, 33\penalty0 (1):\penalty0
  147--180, 2023.

\bibitem[Hu and Lessard(2017{\natexlab{a}})]{hu2017control}
Bin Hu and Laurent Lessard.
\newblock Control interpretations for first-order optimization methods.
\newblock In \emph{2017 American Control Conference (ACC)}, pages 3114--3119.
  IEEE, 2017{\natexlab{a}}.

\bibitem[Hu and Lessard(2017{\natexlab{b}})]{hu2017dissipativity}
Bin Hu and Laurent Lessard.
\newblock Dissipativity theory for nesterov’s accelerated method.
\newblock In \emph{International Conference on Machine Learning}, pages
  1549--1557. PMLR, 2017{\natexlab{b}}.

\bibitem[Huang et~al.(2020)Huang, Geiping, Fowl, Taylor, and
  Goldstein]{huang2020metapoison}
W~Ronny Huang, Jonas Geiping, Liam Fowl, Gavin Taylor, and Tom Goldstein.
\newblock Metapoison: Practical general-purpose clean-label data poisoning.
\newblock \emph{Advances in Neural Information Processing Systems},
  33:\penalty0 12080--12091, 2020.

\bibitem[Ji and Liang(2023)]{liang2023lower}
Kaiyi Ji and Yingbin Liang.
\newblock Lower bounds and accelerated algorithms for bilevel optimization.
\newblock \emph{Journal of machine learning research}, 24\penalty0
  (22):\penalty0 1--56, 2023.

\bibitem[Ji et~al.(2021)Ji, Yang, and Liang]{ji2021bilevel}
Kaiyi Ji, Junjie Yang, and Yingbin Liang.
\newblock Bilevel optimization: Convergence analysis and enhanced design.
\newblock In \emph{International conference on machine learning}, pages
  4882--4892. PMLR, 2021.

\bibitem[Khanduri et~al.(2021)Khanduri, Zeng, Hong, Wai, Wang, and
  Yang]{khanduri2021near}
Prashant Khanduri, Siliang Zeng, Mingyi Hong, Hoi-To Wai, Zhaoran Wang, and
  Zhuoran Yang.
\newblock A near-optimal algorithm for stochastic bilevel optimization via
  double-momentum.
\newblock \emph{Advances in neural information processing systems},
  34:\penalty0 30271--30283, 2021.

\bibitem[Lessard et~al.(2016)Lessard, Recht, and Packard]{lessard2016analysis}
Laurent Lessard, Benjamin Recht, and Andrew Packard.
\newblock Analysis and design of optimization algorithms via integral quadratic
  constraints.
\newblock \emph{SIAM Journal on Optimization}, 26\penalty0 (1):\penalty0
  57--95, 2016.

\bibitem[Li et~al.(2022)Li, Gu, and Huang]{li2022fully}
Junyi Li, Bin Gu, and Heng Huang.
\newblock A fully single loop algorithm for bilevel optimization without
  hessian inverse.
\newblock In \emph{Proceedings of the AAAI Conference on Artificial
  Intelligence}, volume~36, pages 7426--7434, 2022.

\bibitem[Liu et~al.(2022)Liu, Li, Yang, Wai, Hong, Nie, and
  Wang]{liu2022inducing}
Boyi Liu, Jiayang Li, Zhuoran Yang, Hoi-To Wai, Mingyi Hong, Yu~Nie, and
  Zhaoran Wang.
\newblock Inducing equilibria via incentives: Simultaneous design-and-play
  ensures global convergence.
\newblock \emph{Advances in Neural Information Processing Systems},
  35:\penalty0 29001--29013, 2022.

\bibitem[Okuno et~al.(2021)Okuno, Takeda, Kawana, and Watanabe]{okuno2021lp}
Takayuki Okuno, Akiko Takeda, Akihiro Kawana, and Motokazu Watanabe.
\newblock On lp-hyperparameter learning via bilevel nonsmooth optimization.
\newblock \emph{Journal of Machine Learning Research}, 22\penalty0
  (245):\penalty0 1--47, 2021.

\bibitem[Robey et~al.(2023)Robey, Latorre, Pappas, Hassani, and
  Cevher]{robey2023adversarial}
Alexander Robey, Fabian Latorre, George~J Pappas, Hamed Hassani, and Volkan
  Cevher.
\newblock Adversarial training should be cast as a non-zero-sum game.
\newblock \emph{arXiv preprint arXiv:2306.11035}, 2023.

\bibitem[Sehwag et~al.(2020)Sehwag, Wang, Mittal, and Jana]{sehwag2020hydra}
Vikash Sehwag, Shiqi Wang, Prateek Mittal, and Suman Jana.
\newblock Hydra: Pruning adversarially robust neural networks.
\newblock \emph{Advances in Neural Information Processing Systems},
  33:\penalty0 19655--19666, 2020.

\bibitem[Shou and Di(2020)]{shou2020reward}
Zhenyu Shou and Xuan Di.
\newblock Reward design for driver repositioning using multi-agent
  reinforcement learning.
\newblock \emph{Transportation research part C: emerging technologies},
  119:\penalty0 102738, 2020.

\bibitem[Sinha et~al.(2020)Sinha, Khandait, and Mohanty]{sinha2020gradient}
Ankur Sinha, Tanmay Khandait, and Raja Mohanty.
\newblock A gradient-based bilevel optimization approach for tuning
  hyperparameters in machine learning.
\newblock \emph{arXiv preprint arXiv:2007.11022}, 2020.

\bibitem[Zhang et~al.(2022)Zhang, Yao, Ram, Zhao, Chen, Hong, Wang, and
  Liu]{zhang2022advancing}
Yihua Zhang, Yuguang Yao, Parikshit Ram, Pu~Zhao, Tianlong Chen, Mingyi Hong,
  Yanzhi Wang, and Sijia Liu.
\newblock Advancing model pruning via bi-level optimization.
\newblock \emph{Advances in Neural Information Processing Systems},
  35:\penalty0 18309--18326, 2022.

\end{thebibliography}

\end{document}